\documentclass[11pt]{article}
\usepackage{index}
\usepackage{graphicx}
\usepackage{amsfonts,amsmath,amssymb,amsthm,enumerate,latexsym}

\newtheorem{theorem}{Theorem}
\newtheorem{lemma}[theorem]{Lemma}
\newtheorem{corollary}[theorem]{Corollary}

\theoremstyle{definition}

\def\picture#1#2#3{\begin{figure}
\begin{center}
\includegraphics{#2}
\caption{#1}
\def\tmp{#3}
\ifx\tmp\empty\else\label{#3}\fi
\end{center}
\end{figure}}

\def\A{{\mathbb A}}
\def\B{{\mathbb B}}
\def\C{{\mathbb C}}
\def\en{{\mathbb N}}
\def\Prob{\operatorname{Prob}}
\def\CSP{\operatorname{CSP}}
\def\Ext{\operatorname{EXT}}

\let\epsilon\varepsilon

\date{\small Mathematics Subject Classification: 68Q25, 97K30, 97K50, 08A70}
\begin{document}

\title{Complexity of the homomorphism extension problem in the random case}
\author{Alexandr Kazda}

\maketitle
\begin{abstract}
We prove that if $\A$ is a large random relational structure (with at least one
relation of arity at least 2) then the homomorphism extension problem $\Ext(\A)$ is almost surely NP-complete.

\end{abstract}

\section{Introduction}
The complexity of the constraint satisfaction problem (CSP) with a fixed target
structure is a well established field of study in combinatorics and computer
science (see~\cite{nesetril} for an overview). In the last
decade, we have seen algebraic tools brought to bear on the question of CSP
complexity, yielding major new results (see e.g.~\cite{BJK},
\cite{larossetesson-introduction}, \cite{libor}). 

In the algebraic approach, it is customary to study relational structures that
contain all possible constants. If $\A$ is such a structure and we are to
decide the existence of a homomorphism $f:\B\to \A$ then the constant
constraints prescribe values for $f$ at some vertices of $\B$. We are thus
deciding if some partial homomorphism
$f_c:\B \to \A$ can be extended to the whole $\B$. Therefore, $\CSP(\A)$
becomes the \emph{homomorphism extension problem} with target structure $\A$,
denoted by $\Ext(\A)$. It is easy to see that $\CSP(\A)$ reduces to $\Ext(\A)$,
since in CSP we extend the empty partial homomorphism.

In \cite{nesetril-random}, the authors prove that $\CSP(\A)$ is almost surely
NP-complete for $\A$ large random relational structure with at least one at
least binary relation and without loops. We show by a different method that
the same hardness result holds for $\Ext(\A)$ even if we allow loops.

\section{Preliminaries}
A \emph{relational structure} $\A$ is any set $A$ together with a family of
relations $\{R_i:i\in I\}$ where $R_i\subset A^{n_i}$. We call the number
$n_i$ the \emph{arity} of $R_i$. The sequence $(n_i:i\in I)$ determines the
\emph{similarity type} of $\A$. We consider only finite structures (and
finitary relations) in this paper. We use the notation $[n]=\{1,2,\dots,n\}$. 

Let $\A=(A,\{R_i:i\in I\})$ and $\B=(B,\{S_i:i\in I\})$ be two relational structures of the same
similarity type. A mapping $f: A\to B$ is a homomorphism if for every $i\in I$
and every $(a_1,\dots,a_{n_i})\in R_i$ we have $(f(a_1),\dots,f(a_{n_i}))\in
S_i$. 

Let us fix some $p\in(0,1)$ and let $A$ be a set. The relation $S\subset A^l$
is an \emph{$l$-ary random relation} on $A$ if every possible $l$-tuple belongs to $S$
with probability $p$ (independently of other $l$-tuples). We will call any
relational structure with one or more random relations a \emph{random
relational structure}. In particular, a random relational structure with just
one binary relation is a \emph{random digraph}.

The \emph{Constraint Satisfaction Problem} with the target structure $\A$,
denoted by $\CSP(\A)$, consists of deciding whether a given input relational
structure $\B$ of the same similarity type as $\A$ can be homomorphically
mapped to $\A$. It is easy to come up with examples of $\A$ such that
$\CSP(\A)$ is NP-complete and this is in a sense typical behavior
as proved in \cite{nesetril-random}: If $R(n,k)$ is a $k$-ary random relation on the
set $[n]$ (with $p=1/2$) that does not contain any elements of the form $(a,a,\dots,a)$ for
$a\in A$ then
\begin{align}
\forall k\geq 2,\,\lim_{n\to\infty}\Prob(\CSP([n],R(n,k))\text{ is
NP-complete})=1,\label{nesetril1}\\
\forall n\geq 2, \lim_{k\to\infty}\Prob(\CSP([n],R(n,k))\text{ is
NP-complete})=1.\label{nesetril2}
\end{align}

There is a reason why the authors of \cite{nesetril-random} disallow loops: If
$\A$ has only one relation $R$ and $R$ contains a loop $(a,a,\dots,a)$ then
every $\B$ of the same similarity type as $\A$ can be homomorphically mapped to
$\A$ simply by sending everything to $a$, so $\CSP(A)$ is very simple to solve.

Given a target structure $\A$, the \emph{Homomorphism Extension Problem} for
$\A$, denoted by $\Ext(\A)$, consists of deciding whether a given input structure
$\B$ and a given partial mapping $f:\B\to \A$ can be extended to a
homomorphism from $\B$ to $\A$.

Let $\A$ be a set and $a\in A$. The \emph{constant relation} $c_a$ is the unary
relation consisting only of $a$, i.e. $c_a=\{(a)\}$. When searching for a
homomorphism to $\A$, the relation $c_a$ prescribes a set of elements of $B$
that must be mapped to $a$. A little thought gives us that if $\A$ contains 
constant relations for each of its elements (as is usual in the algebraic
treatment of CSP) then $\CSP(\A)$ and $\Ext(\A)$ are
essentially the same problem.

Since the homomorphism extension problem is quite important to algebraists, it
makes sense to ask what is the typical complexity of $\Ext(\A)$. We will use
the phrase ``$\Ext(\A)$ is almost surely NP-compete for $n$ large'' as an
abbreviation
for ``For each $n\in\en$, there exists a random relational structure $\A_n$ (whose precise definition
is obvious from the context) such that we have 
\[
\lim_{n\to\infty}\Prob(\Ext(\A_n)\text{ is NP-complete})=1.\text{''} 
\]

Because additional relations do not make $\CSP$ easier to solve,
the limit (\ref{nesetril1}) gives us that
that $\Ext(\A)$ is almost surely NP-complete if $\A$ is a
large random relational structure with no loops and at least one relation of arity greater than
one. In the remainder of the paper we show that we can allow loops without
making the problem any easier.

\section{The problem $\Ext$ for random digraphs}
We will begin by investigating random digraphs and then generalize our findings
to all relational structures.

\begin{theorem}\label{thmNPgraph}
Let $G$ be a random digraph on $n$ vertices. Then $\Ext(G)$ is almost surely
NP-complete for $n$ large.
\end{theorem}

\begin{proof}

Let $G=(V,E)$ be a digraph. Understand $G$ as a relational structure
and add to $G$ every constant relation possible. Let $v_1,\dots,v_l\in V(G)$. Consider the set 
\[
F_{v_1,\dots,v_l}=\{u\in V(G): \forall i, (v_i,u)\in E(G)\}
\]
We will call this set a \emph{subalgebra} of $G$. 

For an interested reader, we note that sets $F_{v_1,\dots,v_l}$ are 
subalgebras in the universal algebraic sense and our technique can be greatly
generalized to all primitive positive definitions (see \cite{BJK}). For
our proof, however, we need a lot less: Assume that for some choice of
$v_1,\dots, v_l$ the subalgebra $F_{v_1,\dots,v_l}$ induces a loopless triangle
in $G$.
We claim that we can then reduce graph 3-colorability
to $\Ext(G)$, making $\Ext(G)$ NP-complete. 


Let $H$ be a graph whose 3-colorability we wish to test. We then understand $H$
as a symmetric digraph and add to $H$ new vertices $w_1,\dots,w_l$ and new
edges $(w_i,u)$ for each $i\in \{1,\dots,n\}$ and all $u\in V(H)$, obtaining
the digraph $H'$. 
Our $\Ext(G)$ instance will then
consist of the digraph $H'$ along with the partial map $f$ which maps each
$w_i$ to $v_i$. Now $f$ can be extended to a homomorphism if and only if $H$
can be homomorphically mapped into the triangle induced by $F_{v_1,\dots,v_l}$
which happens if and only if $H$ is 3-colorable.


All we need to do now is to show that $G$ almost surely contains a
subalgebra that induces a triangle. Our aim, roughly speaking, is to show that $G$ almost surely
contains many three element subalgebras because then there is a large chance
that at least one of these subalgebras will be a triangle.

We will partition $V(G)$ into two
sets $A=\{1,\dots,\lfloor n/2\rfloor\}$ and $B=\{\lceil n/2
\rceil,\dots,n\}$. 
We will now use points of $A$ to define subalgebras lying in $B$.
Denote by $S_k$ the event ``$G$ contains at least $k$ disjoint three-element
subalgebras of the form $F_{v_1,\dots,v_l}\subset B$ for some $v_1,\dots,v_l\in
A$.''  We
can write 
$$S_k=\bigcup_{\substack{C_1,\dots,C_k\subset B\\
\forall i\neq j,\, C_i\cap C_j=\emptyset\\
\forall i,\, |C_i|=3}} S_{C_1,\dots,C_k},
$$
where $S_{C_1,\dots,C_k}$ is the event ``The sets
$C_1,\dots,C_k$ are subalgebras of $G$''. Finally,
denote by $T_{C_1,\dots,C_k}$ the event ``There exists an $i\in\{1,2.\dots,k\}$
such that the set $C_i$ induces a triangle subgraph of $G$.''

Since a probability that a fixed $C_i$ induces a triangle is $p^6(1-p^3)$, 
the probability of the event $T_{C_1,\dots,C_k}$ is (for $C_1,\dots,C_k$
pairwise disjoint three element sets) 
$$
\Prob(T_{C_1,\dots,C_k})=1-(1-p^6(1-p^3))^k,
$$
which tends to 1 when $k$ goes to infinity.

Observe that the event $S_{C_1,\dots,C_k}$ is independent from the event
$T_{C_1,\dots,C_k}$ for each choice of $C_1,\dots,C_k\subset B$ 
since both events talk about disjoint sets of edges of $G$. 

Assume for a moment that for all $k\in \en$ the value of $\Prob(S_k)$ tends to 1
as $n$ tends to infinity. Then, given an $\epsilon>0$, we choose $k$ so that
$\Prob(T_{C_1,\dots,C_k})\geq 1-\epsilon$. When $n$ is large enough, 
the digraph $G$ contains some $k$ pairwise disjoint
three element subalgebras $C_1,\dots,C_k$
with probability at least $1-\epsilon$. The probability that one of the sets
$C_1,\dots,C_k$ them induces a triangle is $T_{C_1,\dots,C_k}\geq
1-\epsilon$. Thus we get an NP-complete CSP problem with probability at least $(1-\epsilon)^2>
1-2\epsilon$ and since $\epsilon$ was arbitrary, we see that for large $n$ the
homomorphism extension problem is almost surely NP-complete.

It remains to show $\lim_{n\to\infty}\Prob(S_k)=1$ for all $k$. Fix the value
of $k$. For each value of $n$, let $l$ be the integer satisfying $n p^l\geq 1>n p^{l+1}$. We will now search
for the three element subalgebras of $B$ for $n$ large. We proceed in steps: Assume that after $i$
steps we have already found $m$ such subalgebras $C_1,\dots,C_m$. In the $(i+1)$-th step, 
we take the vertices $1+il,2+il,\dots,l+il$
of $A$ and consider the subalgebra $F_{1+il,2+il,\dots,l+il}$. If this
subalgebra lies in $B$, has size three and is disjoint with all the sets $C_1,\dots,C_m$, we
let $C_{m+1}=F_{1+il,2+il,\dots,l+il}$, increase $m$ by one and continue with
the next step. Otherwise, $F_{1+il,2+il,\dots,l+il}$ is not a good candidate for
$C_{m+1}$, so we leave $m$ unchanged and continue with the next step.

What is the probability that we find the $(m+1)$-th
subalgebra in a given step? Every vertex of $G$ is in $F_{1+il,2+il,\dots,l+il}$ with the
probability $p^l$. The probability that $F_{1+il,2+il,\dots,l+il}$ consists of
three yet-unused vertices of $B$ is then equal to
\[
q={|B|-3\cdot m \choose 3}p^{3l}(1-p^l)^{n-3}\geq
\frac{(n/2-3m-3)^3}{6}p^{3l}(1-p^l)^n
\]
If $m\geq k$, we have already won, so assume $m<k$:
\[
q\geq \frac{(n/2-3k)^3}{6}p^{3l}(1-p^l)^n=\frac{(1/2-3k/n)^3}{6}
n^3p^{3l}(1-p^l)^n
\]
Now let $r=\frac{(1/2-3k/n)^3}6$ and observe that $r>0$ for $n$ large
enough. Using the the inequalities $n p^l\geq 1>n p^{l+1}$ we have: 
\[
q \geq r n^3 p^{3l}(1-p^l)^n \geq r (1-p^l)^n> r \left(1-\frac{1}{pn}\right)^n.
\]
The lower
bound on $q$ tends to $r/e^{1/p}$ as $n$ tends to infinity, so there exists a $\delta$ such that $q>\delta>0$
for all $n$ large enough.

Therefore, the probability of producing a new three-element subalgebra in a
given step is at
least $\delta>0$ and this bound does not depend on the number of subalgebras
we have already found. Now observe that $l$
is approximately $\log_{1/p} n$ and therefore we have enough vertices in $A$ for
approximately $s=\frac{n}{2\log_{1/p}n}$ steps. If we choose $n$ large enough, we can have $s$ as
large as we want and so the probability of finding at least $k$ subalgebras
can be arbitrarily close to 1. Therefore, $\lim_{n\to\infty} \Prob(S_k)=1$, concluding the
proof.
\end{proof}

\section{Random relational structures}
It is easy to see that if $\A$ is a relational structure with unary
relations only then $\Ext(\A)$ is always polynomial. We would now like to
investigate the case of relations of arity greater than two. Intuition tells us
that greater arity means greater complexity. The intuition is right.

\begin{lemma}
Let $l\geq 2$, $n$ be large and let $\A=([n],S)$ be a relational structure with
$S$ a random $l$-ary relation. Then
the homomorphism extension problem $\CSP(\A)$ is almost surely NP-complete.
\end{lemma}
\begin{proof}
We have proven the result for $l=2$. If $l>2$, consider the binary relational structure $\B=([n],R)$ where 
$R=\{(x,y)\in [n]^2: (x,y,1,1,\dots,1)\in S\}$. It is easy to see that if $S$ is
a random $l$-ary relation then $\B$ is a random 
digraph where each edge exists with the probability $p$. From
Theorem~\ref{thmNPgraph} we see that
$\Ext(\B)$ is almost surely NP-complete. We will now show how to reduce
$\Ext(\B)$ to $\Ext(\A)$ in polynomial time, proving that $\Ext(\A)$
is almost surely NP-complete.

Using algebraic tools, the reduction of $\Ext(\B)$ to $\Ext(\A)$ follows from the fact that
$R$ is defined by a primitive positive formula that uses only $S$ and the constant $1$.
However, we will provide an elementary reduction here: Let $\C=(C,T)$ be a
relational structure with a single binary relation $T$ and let $f:C\to [n]$ be a 
partial mapping. We add to $C$ a new element $e$, construct the
$l$-ary relation $U=\{(x,y,e,e,\dots,e):(x,y)\in T\}$ and the partial mapping
$g:C\cup \{e\}\to [n]$ so that $g_{|C}=f$ and $g(e)=1$. A little thought gives
us that $g$ can be extended to a homomorphism $(C\cup\{e\},U)\to \A$ if and
only if $f$ can be extended to a homomorphism $(C,T)\to \B$, concluding the
proof.
\end{proof}

Additional relations in $\A$ do not make $\Ext(\A)$ easier, so we have the most
general version of our NP-completeness result:

\begin{corollary}
Let $\A$ be the relational structure $([n],\{R_i:i\in I\})$ where at least one
$R_i$ is a random
relation of arity greater than one. Then $\Ext(\A)$ is 
almost surely NP-complete for $n$ large.
\end{corollary}

As a final note, we will now prove the analogue of the limit~(\ref{nesetril2})
for $\Ext$.

\begin{corollary}
Let us fix a set $A$ of at least two elements and let $\A=(A,R)$ be 
a relational structure with $R$
random $k$-ary relation. Then $\Ext(\A)$ is almost surely NP-complete for $k$
large.
\end{corollary}
\begin{proof}
Assume first that $k$ is even and let $m=k/2$.  

Consider the relational structure $\B=(A^m,S)$ with 
$$
S=\{((a_1,\dots,a_m),(a_{m+1},\dots,a_{2m})):(a_1,\dots,a_{2m})\in R\}.
$$
It is straightforward to prove that $S$ is a binary random relation on $A^m$ and therefore $\Ext(\B)$
is almost surely NP-complete for large even $k$. What is more, $\Ext(\B)$ can
be easily reduced to $\Ext(\A)$: If $\C=(C,T)$ is a relational
structure with $T$ binary and $f:C\to A^m$ is a partial mapping, 
we construct the structure $\C'=(C',T')$ with 
\begin{align*}
C'&=\{(c,i):c\in C, i\in \{1,\dots,m\}\},\\
T'&=\{((c,1),\dots,(c,m),(d,1),\dots,(d,m)):(c,d)\in T\}
\end{align*}
and a partial mapping $g:C'\to A$ such that $g(c,i)=a_i$ whenever $f(c)$ is defined and equal to
$(a_1,\dots,a_m)$.

It is easy to see that $g$ can be extended to a homomorphism from $\C'$ to $\A$
if and only if $f$ can be extended to a homomorphism from $\C$ to $\A$. 

In the case that $k=2m+1$, we fix an $e\in A$, choose
$\B=(A^m,S)$ with
$$
S=\{((a_1,\dots,a_m),(a_{m+1},\dots,a_{2m}):(a_{1},\dots,a_{2m},e)\in R\}
$$
and proceed similarly to the previous case. 

We see that for a large enough $k$, no matter if it is odd or even, the problem
$\Ext(\B)$ is almost surely NP-complete.
\end{proof}

\section{Conclusions}
We have shown that the homomorphism extension problem is almost surely
NP-complete for large relational structures (assuming we have at least one
non-unary relation). In a sense, our result is not surprising since the relational
structures we consider are very dense, so it stands to a reason that we can 
find hard instances most of the time.

It might therefore be interesting to see what is the complexity of CSP or EXT
for large structures obtained by other random processes, particularly when
relations are sparse. Such structures might better correspond to ``typical''
cases of CSP or EXT encountered in practice. Some such results already exist;
see \cite{nesetril-luczak-projective} for a criterion on the random graph
process to almost surely produce projective graphs (if $G$ is
projective then $\Ext(G)$ is NP-complete, see \cite{BJK}). Our guess is that
both CSP and EXT will remain to be almost surely NP-complete in all the
nontrivial cases.

\section{Acknowledgments}
The research was supported by the GA\v CR project GA\v CR 201/09/H012, by the Charles
University project GA UK 67410 and by the grant SVV-2011-263317. The 
author would like to thank Libor Barto for suggesting this problem.

\bibliographystyle{plain}
\bibliography{citations}

\begin{thebibliography}{1}

\bibitem{libor}
Libor Barto and Marcin Kozik.
\newblock Constraint satisfaction problems of bounded width.
\newblock {\em Proceedings of the 50th Annual IEEE Symposium on Foundations of
  Computer Science}, pages 595--603, 2009.

\bibitem{BJK}
Andrei Bulatov, Peter Jeavons, and Andrei Krokhin.
\newblock Classifying the complexity of constraints using finite algebras.
\newblock {\em SIAM J. Comput.}, 34(3):720--742, 2005.
\newblock doi: 10.1137/S0097539700376676.

\bibitem{larossetesson-introduction}
Benoit Larose and Pascal Tesson.
\newblock Universal algebra and hardness results for constraint satisfaction
  problems.
\newblock {\em Theoretical Computer Science}, 410(18):1629--1647, 2009.

\bibitem{nesetril}
Jaroslav~Ne\v set\v ril and Pavol Hell.
\newblock {\em Graphs and homomorphisms}.
\newblock Oxford Lecture Series in Mathematics and Its Applications. Oxford
  University Press, New York, NY, USA, 2004.

\bibitem{nesetril-random}
Jaroslav~Ne\v set\v ril and Tomasz \L{}uczak.
\newblock A probabilistic approach to the dichotomy problem.
\newblock {\em SIAM Journal on Computing}, 36:835--843, 2006.

\bibitem{nesetril-luczak-projective}
Jaroslav~Ne\v set\v ril and Tomasz \L{}uczak.
\newblock When is a random graph projective?
\newblock {\em European Journal of Combinatorics}, 27(7):1147--1154, 2006.
\newblock doi: 10.1016/j.ejc.2006.06.010.

\end{thebibliography}

\end{document}